\documentclass[10pt]{article}%
\usepackage{amsmath}
\usepackage{amsfonts}
\usepackage{mathrsfs}
\usepackage{amssymb}
\usepackage[linkcolor=black,anchorcolor=black,citecolor=black]{hyperref}
\usepackage{graphicx}
\usepackage[body={15.5cm,21cm}, top=3cm]{geometry}%
\providecommand{\U}[1]{\protect\rule{.1in}{.1in}}
\providecommand{\U}[1]{\protect \rule{.1in}{.1in}}
\newtheorem{theorem}{Theorem}[section]

\newtheorem{corollary}[theorem]{Corollary}

\newtheorem{definition}[theorem]{Definition}

\newtheorem{lemma}[theorem]{Lemma}

\newtheorem{remark}[theorem]{Remark}

\newenvironment{proof}[1][Proof]{\noindent \textbf{#1.} }{\  \rule{0.5em}{0.5em}}
\begin{document}

\title{Some sample path properties of $G$-Brownian motion}

\author{Falei Wang\thanks{School of Mathematics, Shandong University; flwang2011@gmail.com.}
\and  Guoqiang Zheng\thanks{School of Mathematics, Shandong University;  zhengguoqiang.ori@gmail.com. Wang and Zheng's research was
partially supported by NSF (No. 10921101) and by the 111
Project (No. B12023)}}
\date{}

\maketitle

\begin{abstract}
In this paper, we shall study the basic absolute properties  of $G$-Brownian motion, i.e., those properties which hold for q.s. $\omega$.
These include the characterization of the zero set and the local maxima of the $G$-Brownian motion paths.
The key ingredient of our approach is an estimate of solutions to $G$-heat equations.
We also show that the $G$-Brownian path is nowhere  locally H\"{o}lder continuous of order $\gamma$ for any $\gamma> \frac{1}{2}$.
\end{abstract}

\section{Introduction}

Brownian motion was rigorously constructed to model the random walk of pollen particles in a liquid.
After decades of development, Kac \cite{K} established the connection between the expectations of stochastic functionals and the solutions of partial differential equations (PDEs for short), which is  the famous Feynnman-Kac formula.
Impressed  by this bridging relationship, Peng  introduced the $G$-expectation theory (see  \cite{Peng2004,Peng2005,Peng4})
through the following $G$-heat equation:
\begin{align*}
\partial_{t}u(t,x)-G(\partial^2_{xx}u(t,x))=0.
\end{align*}
Almost immediately, the notion of $G$-Brownian motion and the corresponding stochastic calculus of It\^{o}'s type were also established.
For a detailed account and recent development of this theory we refer the reader to  \cite{G,LQ,L}.

One important feature of $G$-expectation theory is that, the random variables are obliged to be quasi-continuous to be in $\mathbb{L}_{G}^1(\Omega)$ from \cite{DHP}.
So one has to be careful when dealing with issues involving integrability, stopping times etc.
Since the indicator  functions of Borel sets are of priority in   Lebesgue integration theory,
an interesting question is whether the indication  functions  are in $\mathbb{L}_{G}^1(\Omega)$.

In this paper, we shall prove the set $\{\omega: B_t=a\}$ is a polar set for each $t>0$ and $a\in\mathbb{R}$.
The proof is base on an estimate of solutions to $G$-heat equations, which is a useful tool for the study of $G$-expectation.
In particular,  we  obtain that  the indicator function $\mathbf{1}_{B_t\in O}\in \mathbb{L}_G^1(\Omega_t)$, where $O$ is  some ``regular'' Borel set of $\mathbb{R}$.
Moreover, for q.s. $\omega$, the $G$-Brownian motion path is  monotonic in no  time interval and the set of local maximum points  for $G$-Brownian  motion is dense.
In addition, the $G$-Brownian path is  nowhere locally H\"{o}lder continuous of order $\gamma$ for any $\gamma>\frac{1}{2}$ as the classical case.

This paper is organized as follow. In section 2, we recall some necessary notations and results of $G$-expectation theory.  In section 3,  we state our main results and some applications are given.

\section{Preliminaries}
The main purpose of this section is to recall some preliminary
results in $G$-framework which are needed in the sequel.  More
details can be found in  Denis et al \cite{DHP} and Peng \cite{Peng4}.

Let $\Omega=C_{0}(\mathbb{R}^{+})$ be the space of all $\mathbb{R}%
$-valued continuous paths $(\omega_{t})_{t\geq0}$, with $\omega_{0}=0$,
equipped with the distance%
\[
\rho(\omega^{1},\omega^{2}):=\sum_{i=1}^{\infty}2^{-i}[(\max_{t\in \lbrack
0,i]}|\omega_{t}^{1}-\omega_{t}^{2}|)\wedge1].
\]
Denote by $\mathcal{B}({\Omega})$  the Borel $\sigma$-algebra of $\Omega$.
For each $t\in[0,\infty)$, we also introduce the following spaces.
\\
$\bullet$ $\Omega_{t}:=\{\omega({\cdot\wedge t}):\omega\in\Omega\}$, $\mathcal{F}_{t}:=\mathcal {B}(\Omega_{t})$,
\\
$\bullet$ $L^{0}{(\Omega)}:$ the space of all $\mathcal
{B}(\Omega)$-measurable real functions,
\\
$\bullet$ $L^{0}{(\Omega_{t})}:$ the space of all $\mathcal{F}_{t}$-measurable real functions,
\\
$\bullet$
$C_{b}{(\Omega)}:$ all  continuous  bounded elements in $L^{0}{(\Omega)}$, $C_{b}{(\Omega_{t})}:=C_{b}{(\Omega)}\cap L^{0}{(\Omega_{t})}.$

In Peng \cite{Peng4}, a $G$-Brownian motion is constructed on a sublinear expectation
space $(\Omega,\mathbb{L}_{G}^1,\hat{\mathbb{E}},(\hat{\mathbb{E}}_t)_{t\geq 0})$, where   $\mathbb{L}_{G}^p(\Omega)$
 is a Banach space
under the natural norm $\|X\|_p=\hat{\mathbb{E}}[|X|^p]^{1/p}$.   In this space the corresponding
canonical process $B_t(\omega) = \omega_t$
 is a $G$-Brownian motion.
Denis et al.\cite{DHP}  proved that
$L^{0}{(\Omega)}\supset\mathbb{L}_{G}^{p}(\Omega)\supset
C_{b}{(\Omega)}$ and there exists a weakly compact family $\mathcal
{P}$ of probability measures defined on $(\Omega, \mathcal
{B}(\Omega))$ such that
$$\hat{\mathbb{E}}[X]=\sup\limits_{P\in\mathcal{P}}E_{P}[X], \ \ X\in\mathbb{L}_{G}^{1}{(\Omega)}.$$

\begin{remark}\label{5}{\upshape Denis et al. \cite{DHP} gave a concrete set $\mathcal{P}_M$ that represents $\hat{\mathbb{E}}$.
Consider a 1-dimensional Brownian motion $ {W_t}$ on $(\Omega,\mathcal{F},P)$, then
\[
\mathcal{P}_M := \{P_{\theta} : P_{\theta}= P\circ X^{-1},\ X_t = \int^t_0 \theta_sdW_s,\  \theta\in L^2_{\mathcal{F}}([0, T ]; [\underline{\sigma}^2, \overline{\sigma}^2])\}\]
is a set that represents $\hat{\mathbb{E}}$, where $L^2_{\mathcal{F}}([0, T ]; [\underline{\sigma}^2, \overline{\sigma}^2])$ is the collection of all $\mathcal{F}$-adapted
measurable processes with $\underline{\sigma}^2 \leq |\theta_s|^2 \leq\overline{\sigma}^2$.
}
\end{remark}

Now we introduce the Choquet capacity:
$$c(A):=\sup\limits_{P\in\mathcal{P}}P(A), \ \ A\in\mathcal
{B}(\Omega).$$
\begin{definition}A set $A\subset\mathcal{B}(\Omega)$ is polar if $c(A)=0$.  A
property holds $``quasi$-$surely''$ (q.s.) if it holds outside a
polar set.
\end{definition}
\begin{definition}A real function $X$ on $\Omega$ is said to be quasi-continuous if for each $\varepsilon>0$,
there exists an open set $O$ with $c(O)<\varepsilon$ such that
$X|_{O^{c}}$ is continuous.
\end{definition}

\begin{definition}  We say that  $X:\Omega\mapsto\mathbb{R}$ has a quasi-continuous
version if there exists a quasi-continuous function $Y:\Omega\mapsto\mathbb{R}$ such
that $X = Y$, q.s..
\end{definition}

 Then $\mathbb{L}_{G}^{p}(\Omega)$ can be characterized as
follows:
$$\mathbb{L}_{G}^{p}(\Omega)=\{X\in {L}^{0}(\Omega)|\lim\limits_{N\rightarrow\infty}\mathbb{\hat{E}}[|X|^p\mathbf{1}_{|X|\geq N}]=0\ \text{and}\  X\ \text {has a quasi-continuous version}\}.$$

\section{The main results}
In this paper, we always assume $\underline{\sigma}^2>0.$ Without loss of generality, assume $\overline{\sigma}^2=1$.
This section is devoted to study the sample path properties of $G$-Brownian motion, before that we need some estimates  from the theory of fully nonlinear partial differential equations.
\begin{lemma}\label{lemma1}
Suppose $u^n$ is the solution of the following $G$-heat equation,
\begin{align}\label{w1}
\begin{cases}
&
\partial_{t}u^n(t,x)-G(\partial^2_{xx}u^n(t,x))=0, \ (t,x)\in(0,T)\times \mathbb{R},\\
& u^n(0,x)=\exp\{-\frac{n(x-a)^2}{2}\},
\end{cases}
\end{align}
where $n\geq 0$ and $a\in\mathbb{R}$.
Then there exists some constant $\alpha\in(0,\frac{1}{2}]$  depending only on $G$  such that
\[
u^n(t,x)\leq \frac{1}{(nt+1)^\alpha}\exp\{-\frac{ n(x-a)^2}{2(nt+1)}\}.
\]\end{lemma}
\begin{proof}
Set $\displaystyle v^n(t,x)=\frac{1}{(nt+1)^\alpha}\exp\{-\frac{n(x-a)^2}{2(nt+1)}\}$, here $\alpha$  is a positive constant which need to be determined in the following proof.
Then we get
\begin{align*}
&\partial_tv^n(t,x)=-\frac{n\alpha}{nt+1}v(t,x)+\frac{ n^2(x-a)^2}{2(nt+1)^2}v(t,x),\
\partial_xv^n(t,x)=-\frac{ n(x-a)}{nt+1}v(t,x),\\
&\partial_{xx}^2v^n(t,x)=-\frac{n }{nt+1}v(t,x)+\frac{ n^2(x-a)^2}{(nt+1)^2}v(t,x).
\end{align*}
Consequently,
\begin{align*}
&\partial_tv^n(t,x)-G(\partial_{xx}^2v^n(t,x))\\& \ \
=-\frac{n\alpha}{nt+1}v^n(t,x)+\frac{ n^2(x-a)^2}{2(nt+1)^2}v^n(t,x)-G(-\frac{n }{nt+1}v^n(t,x)+\frac{ n^2(x-a)^2}{(nt+1)^2}v^n(t,x))\\
&\ \ \geq -\frac{n\alpha}{nt+1}v^n(t,x)+\frac{ n^2(x-a)^2}{2(nt+1)^2}v^n(t,x)-G(-\frac{n }{nt+1}v^n(t,x))-G(\frac{ n^2(x-a)^2}{(nt+1)^2}v^n(t,x))\\
&\ \ =\frac{n}{2(nt+1)}(\underline{\sigma}^2-2\alpha)v^n(t,x)+\frac{n^2(x-a)^2}{2(nt+1)^2}(1-\overline{\sigma}^2) v^n(t,x)\\
&\ \ \geq 0,
\end{align*}
where $\alpha=\frac{1}{2}\underline{\sigma}^2\leq \frac{1}{2}$.
Thus $v^n(t,x)$ is a bounded supersolution of PDE \eqref{w1} for each $n\geq 0$. From the comparison theorem (see Appendix C in Peng \cite {Peng4}), we derive
\[
u^n(t,x)\leq v^n(t,x),
\]
which completes the proof.
\end{proof}

\begin{remark}
{\upshape
We remark that the above results is non-trivial. Indeed if $\underline{\sigma}^2<\overline{\sigma}^2$, then the constant $ \alpha$ appearing in the above results is strictly less that $\frac{1}{2}$ because of the nonlinearity.
}
\end{remark}

Now we shall  consider some simple Borel functions on $\Omega$.
\begin{lemma} \label{w2}
For each $a\in\mathbb{R}$, $\mathbf{1}_{B_t\leq a}, \mathbf{1}_{B_t<a}\in \mathbb{L}_G^1(\Omega_t)$.
\end{lemma}
\begin{proof}
Without loss of generality, assume $a=0$.
Denote $\varphi^n(x)=\mathbf{1}_{x\leq 0} +\exp(-nx^2)\mathbf{1}_{x>0}$ for each $n\geq 1$. It is obvious that $\varphi^n(B_t)\in\mathbb{L}_G^1(\Omega_t).$
Note that
\[
\mathbb{\hat{E}}[\varphi^n(B_t)-\mathbf{1}_{B_t\leq 0}]\leq \mathbb{\hat{E}}[\exp(-nB_t^2)].
\]

Then consider the following $G$-heat equation,
\begin{align*}
\begin{cases}
&
\partial_{t}u^n(t,x)-G(\partial^2_{xx}u^n(t,x))=0, (t,x)\in(0,T)\times \mathbb{R},\\
& u^n(0,x)=\exp\{-nx^2\}.
\end{cases}
\end{align*}
Applying Lemma \ref{lemma1} and  nonlinear Feynman-Kac formula in Peng \cite{Peng4}, we get,
\[
\mathbb{\hat{E}}[\exp(-nB_t^2)]=u^n(t,0)\leq \frac{1}{n^\alpha t^\alpha},
\]
 where $\alpha$ is given in Lemma \ref{lemma1}.
Thus  sending $n\rightarrow\infty$, we have\[
\lim\limits_{n\rightarrow\infty}\mathbb{\hat{E}}[\varphi^n(B_t)-\mathbf{1}_{B_t\leq 0}]\leq \lim\limits_{n\rightarrow\infty}\frac{1}{n^{\alpha}t^\alpha}=0,
\]
and we deduce $\mathbf{1}_{B_t\leq 0}\in \mathbb{L}_G^1(\Omega_t)$.

By a similar analysis, we also have \[
\mathbb{\hat{E}}[\mathbf{1}_{B_t=0}]\leq \lim\limits_{n\rightarrow\infty}\mathbb{\hat{E}}[\exp(-nB_t^2)]=0. \]
Thus $\mathbf{1}_{B_t\leq 0}= \mathbf{1}_{B_t< 0}, $ q.s. and the proof is complete.
\end{proof}
\begin{remark}{\upshape
If $\underline{\sigma}^2=0$, it follows from Remark \ref{5} that $c(B_t=0)=1$ and $\mathbf{1}_{B_t< 0}$  is not in $\mathbb{L}_G^1(\Omega_t)$.
}
\end{remark}

 Unlike the classical case, we can not get $\mathbf{1}_{H}\in \mathbb{L}_G^1(\Omega_t)$ for any  Borel set $H\in\mathcal{F}_t$. Indeed Soner et al \cite{STZ} constructed a counterexample. However, from Lemma \ref{w2}, we immediately have the following theorem.
\begin{theorem}\label{2}
For the $G$-Brownian motion $B$, we have the following properties.
\begin{description}
\item[(i)] Given a Borel set $O$ of $\mathbb{R}$. If there exists a sequence points $\{a_i\}_{i=1}^{\infty}$  such that
$\partial O\subset \cup_{i=1}^{\infty} a_{i}$, then $\mathbf{1}_{B_t\in O}\in \mathbb{L}_G^1(\Omega_t)$.
\item[(ii)] For each $a\in\mathbb{R}$, the set $Z=\{t: \ B_t=a\}$ is  q.s. closed and has zero Lebesgue measure.
\end{description}
\end{theorem}
\begin{proof}
(i)
From Lemma \ref{w2}, for each $a\in\mathbb{R}$, $\mathbf{1}_{B_t=a}=0$, q.s..
Then \[c(B_t\in \partial O)\leq \sum\limits_{i=1}^{\infty}c(B_t =a_i)=0.\]
Note that the set $\{\omega: B_t\in \partial O\}$ is a closed subset of $\Omega$.
By Lemma 3.4 of Song \cite{Song}, we deduce that  $\mathbf{1}_{B_t\in O}$ is quasi-continuous.
Recalling the pathwise description of $\mathbb{L}_G^1(\Omega_t)$, we get $\mathbf{1}_{B_t\in O}\in \mathbb{L}_G^1(\Omega_t)$.

(ii)
From the continuity of $G$-Brownian motion paths, we get  $Z$ is closed.
Recalling $c(B_t=a)=0$ for each $t>0$, we obtain
\[
\mathbb{\hat{E}}[\int^{\infty}_0\mathbf{1}_Z(s)ds]\leq \int^{\infty}_0\mathbb{\hat{E}}[\mathbf{1}_Z(s)]ds=0.
\]
It follows that $Z$ has q.s. zero Lebesgue measure, which is the desired result.
\end{proof}

We also have the following result.
\begin{lemma}\label{3}
For each $t>0$, \[
c(B_t^*=0)=0,
\]
where $B^*_t=\sup\limits_{s\in[0,t]}B_s.$
\end{lemma}
\begin{proof}
Denote $\psi^n(x)=\exp\{-nx^2\}$ for each $n$. By Remark \ref{5}, we have
\[
\hat{\mathbb{E}}[\psi^n(B^*_t)]=\sup\limits_{\theta}E_P[\psi^n(M^{\theta,*}_t)],
\]
where $M^{\theta}_t=\int_{0}^{t}\theta_udW_u$.
From the time transformation for continuous martingale (see
Theorem V.1.6 in Revuz and Yor \cite{MY}), there exists a Brownian motion $\hat{W}$ on $(\Omega, \mathcal{F}, P)$
(or an enlargement of this probability space) such that
 \[
\hat{W}_{\langle M^{\theta}\rangle_t}=M^{\theta}_t.
\]
Since $\theta^2\in[\underline{\sigma}^2,\overline{\sigma}^2]$, we have
\[
\int^t_0\theta^2_sds\geq \underline{\sigma}^2t.
\]
Thus $\hat{W}^*_{\langle M^{\theta}\rangle_t}\geq \hat{W}^*_{\underline{\sigma}^2t}$ for each $\theta$.
Therefore we obtain
\[
\hat{\mathbb{E}}[\psi^n(B^*_t)]\leq E_P[\psi^n(\hat{W}^*_{\underline{\sigma}^2t})]=E_P[\psi^n(|\hat{W}_{\underline{\sigma}^2t}|)],
\]
since $\hat{W}_t^*\overset{d}{=}|\hat{W}_t|$.
From Lemma \ref{lemma1}, we get
\[
c(B_t^*=0)\leq\lim\limits_{n\rightarrow\infty}\hat{\mathbb{E}}[\psi^n(B^*_t)]=0,
\]
which is the desired result.
\end{proof}

Under  the framework of $G$-expectation,  independence is noncommutative and it requires that the test functions  are  bounded Lipschitz.
In the next lemma we will prove the consistency of  independence notion under $G$-expectation with the classical independence notion in some sense.
\begin{lemma}\label{1}
For each open subset $O$, closed subset $F$ of $\mathbb{R}$ and $t,s\geq 0$,
\begin{align*}
&c(B_t\in O,B_{t+s}-B_t\in O)=c(B_t\in O)c(B_{t+s}-B_t\in O),\\
&c(B_t\in F,B_{t+s}-B_t\in F)=c(B_t\in F)c(B_{t+s}-B_t\in F).
\end{align*}
\end{lemma}
\begin{proof}
For each $n$, denote $\psi^n(x)=\frac{nd(x,\mathbb{R}/O)}{1+nd(x,\mathbb{R}/O)}$. Then $\psi^n(x)$ is a Lipschitz function and $\psi^n\uparrow \mathbf{1}_O.$
Then applying Theorem 1.10 of Chapter VI in  Peng \cite{Peng4}, we have
Thus
\[
\mathbb{\hat{E}}[\mathbf{1}_O(B_t)]=\lim\limits_{n\rightarrow\infty}\mathbb{\hat{E}}[\psi^n(B_t)].\]
Consequently,
\begin{align*}
\mathbb{\hat{E}}[\mathbf{1}_O(B_t)\mathbf{1}_O(B_{t+s}-B_t)]=\lim\limits_{n\rightarrow\infty}\mathbb{\hat{E}}[\psi^n(B_t)\psi^n(B_{t+s}-B_t)]
=&\lim\limits_{n\rightarrow\infty}\mathbb{\hat{E}}[\psi^n(B_t)]\lim\limits_{n\rightarrow\infty}\mathbb{\hat{E}}[\psi^n(B_{t+s}-B_t)]\\
=&\mathbb{\hat{E}}[\mathbf{1}_O(B_t)]\mathbb{\hat{E}}[\mathbf{1}_O(B_{t+s}-B_t)],
\end{align*}
and the first equality holds true.
The second equality can be proved in a similar way.
\end{proof}

Now we introduce the following nonlinear PDE:
\begin{align}\label{w3}
\begin{cases}
&
\partial_{t}u^n(t,x)-G(\partial^2_{xx}u^n(t,x))=0, (t,x)\in(0,T)\times \mathbb{R},\\
& u^n(0,x)=\varphi^n(x),
\end{cases}
\end{align}
where function $\varphi^n$ is given in the proof of Lemma \ref{w2}. It is obvious $u^{n+1}\leq u^n$ for each $n$.

By the interior regularity of $u^n$ (see Krylov \cite{K1} and Wang \cite{WL2}), there exists a constant $\gamma\in(0,1)$ depending on $G$ such that for each $\epsilon>0$, we can find some constant $C$ depending only on $\epsilon,\gamma$ and $G$ so that
\[
\|u^n\|_{C^{1+\frac{\gamma}{2},2+\gamma}([\epsilon,T]\times\mathbb{R})}<C.
\]

Then for each $\epsilon$, there exists a subsequence $\{ u^{n'}\}_{n'=1}^{\infty} $ such $\partial_t u^{n'}, \partial_x u^{n'}$ and  $\partial_{xx}u^{n'}$
are Cauchy sequences in $[\epsilon,T]\times\mathbb{R}$. Denoting $v(t,x):=\lim\limits_{n\rightarrow\infty}u^n(t,x)$,
we can get $\partial_t u^{n'}, \partial_x u^{n'}, \partial_{xx}u^{n'}$ converge respectively to $ \partial_t v, \partial_x v, \partial_{xx}v$ in $[\epsilon,T]\times\mathbb{R}$.
Since $\epsilon$ is arbitrary, $v(t,x)\in C^{1,2}((0,T]\times\mathbb{R})$.
Moreover, $v(t,x)$ is a solution of the following nonlinear PDE:
\begin{align*}
\begin{cases}
&
\partial_{t}v(t,x)-G(\partial^2_{xx}v(t,x))=0, (t,x)\in(0,T)\times \mathbb{R},\\
& v(0,x)=\mathbf{1}_{x\leq 0}.
\end{cases}
\end{align*}
In addition,  \[v(t,x)=\lim\limits_{n\rightarrow\infty}u^n(t,x)=\lim\limits_{n\rightarrow\infty}\mathbb{\hat{E}}[\varphi^n(B_t+x)]=\mathbb{\hat{E}}[\mathbf{1}_{B_t+x\leq 0}],\]
which is the nonlinear Feynman-Kac formula.

\begin{lemma}\label{lemma2}
For each $t>0$, $v(t,\cdot):\mathbb{R}\mapsto (0,1)$ is a strictly decreasing  function.
\end{lemma}
\begin{proof}
Applying Theorems 10 and 31 in Denis, Hu and Peng \cite{DHP} yields that
\[
\lim\limits_{x\rightarrow +\infty}v(t,x)=0,\ \lim\limits_{x\rightarrow -\infty}v(t,x)=1.
\]
For each $x<y$, we have
\begin{align*}
v(t,x)-v(t,y)&=\mathbb{\hat{E}}[\mathbf{1}_{B_t+x\leq 0}]-\mathbb{\hat{E}}[\mathbf{1}_{B_t+y\leq 0}]
\geq -\mathbb{\hat{E}}[-\mathbf{1}_{-y<B_t\leq -x}]>0.
\end{align*}
The last estimate comes from Corollary 3.5 of  Li \cite{LX} and we establish the desired result.
\end{proof}

Now we shall study the    local
maxima of the $G$-Brownian motion paths.

\begin{definition}
Given a function $f:[0,\infty)\mapsto\mathbb{R}$. A number $t$ is called a point of local maximum, if there exists a number $\delta>0$ such that
\[
f(s)\leq f(t), \ \ \forall s\in[(t-\delta)^+,t+\delta].
\]
\end{definition}
\begin{theorem}\label{4}
 For q.s. $\omega$, the $G$-Brownian motion $B$ is  monotone in no interval.
\end{theorem}
\begin{proof}
Denote by $F$ the set of all paths $\omega$ with the property that $B_t(\omega)$  as a function of $t$ is  monotone in some time interval. Then we have
\[
F=\underset{s,t\in Q}{\cup}\{\omega: B\ \text{is  monotone in}\  [s,t]\},
\]
where $Q$ is the set of all rational points in $[0,\infty)$. Therefore it  suffices to show that in any such
interval, say in $[0,1]$,   the path $B$ is q.s. not monotone.
By virtue  of the symmetry of $G$-Brownian motion, it  suffices to prove the set
\[
A=\{\omega: B\ \text{is  decreasing on}\  [0,1]\}
\]
is a polar set. But $A=\cap_{n}A_n$, where
\[
A_n=\cap^{n}_{i=1}\{\omega: B_{\frac{i}{n}}-B_{\frac{i-1}{n}}\leq 0\}.
\]
From the Lemmas \ref{1} and \ref{lemma2}, we conclude that
\[
c(A_n)=\Pi^{n}_{i}c(\{\omega: B_{\frac{i}{n}}-B_{\frac{i-1}{n}}\leq 0\})=\rho^n,
\]
where $\rho=\mathbb{\hat{E}}[\mathbf{1}_{B_t\leq 0}]=\mathbb{\hat{E}}[\mathbf{1}_{B_s\leq 0}]<1$ for each $s,t$.
Thus $c(A)\leq \lim\limits_{n\rightarrow\infty}c(A_n)=0,$ which completes the proof.
\end{proof}

Note that if $f$ is a continuous function which is monotone in no interval, then set of points of local maximum for $f$  is dense (see Chapter 2.9 in \cite{KS}).
Thus the following corollary is a direct result of theorem \ref{4}.
\begin{corollary}
The set of points of local maximum for the $G$-Brownian motion $B$ is  dense in $[0,\infty)$ q.s..
\end{corollary}

Then we shall show the nowhere differentiability of $G$-Brownian paths.
\begin{lemma}\label{l}
For each $t>0$, $a\in\mathbb{R}$ and $\epsilon>0$, we have
\[
c(|B_t-a|\leq \epsilon)\leq \exp(\frac{1}{2})\frac{\epsilon^{2\alpha}}{t^{\alpha}},
\]
where $\alpha=\frac{\underline{\sigma}^2}{2}$.
\end{lemma}
\begin{proof}
By Lemma \ref{lemma1}, we obtain for each $a\in\mathbb{R}$ and $n\geq 0$,
\[
\mathbb{\hat{E}}[\exp(-\frac{n(B_t-a)^{2}}{2})]\leq \frac{1}{(1+nt)^{\alpha}}.
\]
Thus for each $n$ and $\epsilon>0$, we obtain that
\[
\hat{\mathbb{E}}[\mathbf{1}_{|B_t-a|\leq \epsilon}]
\leq \exp(\frac{n\epsilon^2}{2})\mathbb{\hat{E}}[\exp(-\frac{n(B_t-a)^{2}}{2})]\leq \exp(\frac{n\epsilon^2}{2})\frac{1}{(1+nt)^{\alpha}}.
\]
Therefore, taking $n=\frac{1}{\epsilon^2}$ yields that,
\[
c(|B_t-a|\leq \epsilon)\leq \exp(\frac{1}{2})\frac{\epsilon^{2\alpha}}{t^{\alpha}},
\]
which completes the proof.
\end{proof}
\begin{theorem}
The $G$-Brownian path is q.s. nowhere locally H\"{o}lder continuous of order $\gamma$ for any $\gamma>\frac{1}{2}.$ In particular, the $G$-Brownian path is nowhere differentiable  and has infinite variation on any interval.
\end{theorem}
\begin{proof}
We shall only have to prove
that the $G$-Brownian path is q.s. nowhere locally H\"{o}lder continuous of order $\gamma$ on interval $[0,1]$ as the classical case.
Fix a  $\gamma>\frac{1}{2}$,  we can find an integer $l$ so that $l>\frac{1}{\alpha(2\gamma-1)}.$
For each $\beta>0$, we consider
$$ A_n=\{\omega: \ \text{there is a point $s\in[0,1]$, such that $|B_t(\omega)-B_s(\omega)|<\beta| t-s|^\gamma$ whenever  $|t-s|<\frac{l}{n}$}\}.$$
It is obvious that $A_n$ is increasing to $\cup_{n=1}^{\infty}A_n=\lim\limits_{n\rightarrow\infty}A_n.$
Note that any locally $\gamma$-H\"{o}lder continuous path $(\omega_s)_{s\in[0,1]}$   belongs to some $A_n$ with some $\beta>0$.
Then it suffice to prove $ c(\cup_{n=1}^{\infty}A_n)$ is polar for each $\beta>0$.

For each $0\leq k\leq n$, we set
$$ Z_k:=\max\limits_{1\leq i\leq l} | B_{\frac{k+i}{n}}-B_{\frac{k+i-1}{n}}| \ \ \text{and}\ \
D_n:=\{\omega: \ \text{ there is a constant $k$ such that}\ \ Z_k(\omega)\leq 2\beta\frac{l^\gamma}{n^\gamma}\}.$$
It is easy to check that $A_n\subset D_n$. Indeed, for each $\omega \in A_n$ that is locally  H\"{o}lder continuous  at $s$, if $k$ is chosen to be the largest integer  such that $k\leq ns,$ then
$$Z_k(\omega)\leq 2\beta\frac{l^\gamma}{n^\gamma}.$$
Thus  from Theorem 1.10 of Chapter VI in  Peng \cite{Peng4}, we get
$$ c(\cup_{n=1}^{\infty}A_n)=\lim\limits_{n\rightarrow\infty} c(A_n)\leq\liminf\limits_{n\rightarrow\infty} c(D_n)\leq \liminf\limits_{n\rightarrow\infty} \sum_{k=0}^{n}c(Z_k(\omega)\leq 2\beta\frac{l^\gamma}{n^\gamma}), $$
 In the spirit of  Lemma \ref{1}, we obtain,
\[\sum_{k=0}^{n}c(Z_k(\omega)\leq 2\beta\frac{l^\gamma}{n^\gamma})=n  c(Z_0(\omega)\leq 2\beta\frac{l^\gamma}{n^\gamma})\leq n  c(\vert B_{\frac{1}{n}}\vert\leq 2\beta\frac{l^\gamma}{n^\gamma})^l\leq n c(\vert B_{n^{2\gamma-1}}\vert\leq 2\beta {l^\gamma}) ^l. \]
Then by Lemma \ref{l}, we get that
$$
c(\cup_{n=1}^{\infty}A_n)\leq  \liminf\limits_{n\rightarrow\infty}n c(\vert B_{n^{2\gamma-1}}\vert\leq 2\beta {l^\gamma})^l\leq \liminf\limits_{n\rightarrow\infty}n[\exp(\frac{1}{2})\frac{{(2\beta l^\gamma)}^{2\alpha}}{{n^{(2\gamma-1)\alpha}}}]^l=0,
$$ since $l>\frac{1}{\alpha(2\gamma-1)}$. The proof is complete.
\end{proof}

\end{document}